\documentclass{article}
\usepackage[T1]{fontenc}
\RequirePackage[OT1]{fontenc}
\usepackage{amsmath,amsfonts,amssymb,amsthm,psfrag, color,wrapfig,float}
\usepackage{authblk}

\usepackage[francais,english]{babel}
\title{Mean-field limit versus small-noise limit for some interacting particle systems}
\author[1]{Samuel Herrmann}
\author[2]{Julian Tugaut}
\affil[1]{Institut de Math\'ematiques de Bourgogne (IMB) - UMR 5584, Universit\'e de Bourgogne, B.P. 47
870, 21078 Dijon Cedex, France}
\affil[2]{Universit\'e Jean Monnet, Institut Camille Jordan, 23, rue du docteur Paul Michelon,
CS 82301,
42023 Saint-\'Etienne Cedex 2,
France.}

\begin{document}
\date{July 4, 2014}

\newcommand{\bRb}{\mathbb{R}}
\newcommand{\bCb}{\mathbb{C}}
\newcommand{\bEb}{\mathbb{E}}
\newcommand{\bKb}{\mathbb{K}}
\newcommand{\bQb}{\mathbb{Q}}
\newcommand{\bFb}{\mathbb{F}}
\newcommand{\bGb}{\mathbb{G}}
\newcommand{\bNb}{\mathbb{N}}
\newcommand{\bZb}{\mathbb{Z}}

\newcommand{\deriv}{\stackrel{\mbox{\bf\Large{}$\cdot$\normalsize{}}}}
\newcommand{\dederiv}{\stackrel{\mbox{\bf\Large{}$\cdot\cdot$\normalsize{}}}}

\theoremstyle{plain} \newtheorem{thm}{Theorem}[section]
\theoremstyle{plain} \newtheorem{prop}[thm]{Property}
\theoremstyle{plain} \newtheorem{props}[thm]{Properties}
\theoremstyle{plain} \newtheorem{ex}[thm]{Example}
\theoremstyle{plain} \newtheorem{contrex}[thm]{Coounterexample}
\theoremstyle{plain} \newtheorem{cor}[thm]{Corollary}
\theoremstyle{plain} \newtheorem{hyp}[thm]{Hypothesis}
\theoremstyle{plain} \newtheorem{hyps}[thm]{Hypotheses}
\theoremstyle{plain} \newtheorem{lem}[thm]{Lemma}
\theoremstyle{plain} \newtheorem{rem}[thm]{Remark}
\theoremstyle{plain} \newtheorem{nota}[thm]{Notation}
\theoremstyle{plain} \newtheorem{defn}[thm]{Definition}

\newcommand{\cRc}{\mathcal{R}}
\newcommand{\cCc}{\mathcal{C}}
\newcommand{\cEc}{\mathcal{E}}
\newcommand{\cKc}{\mathcal{K}}
\newcommand{\cQc}{\mathcal{Q}}
\newcommand{\cFc}{\mathcal{F}}
\newcommand{\cGc}{\mathcal{G}}
\newcommand{\cNc}{\mathcal{N}}
\newcommand{\cZc}{\mathcal{Z}}
\newcommand{\cOc}{\mathcal{O}}
\newcommand{\cSc}{\mathcal{S}}
\newcommand{\cAc}{\mathcal{A}}
\newcommand{\cBc}{\mathcal{B}}
\newcommand{\cIc}{\mathcal{I}}
\newcommand{\cDc}{\mathcal{D}}
\newcommand{\cLc}{\mathcal{L}}
\newcommand{\cHc}{\mathcal{H}}

\newcommand{\Ima}{{\rm Im}}
\newcommand{\Rea}{{\rm Re}}
\newcommand{\gaga}{\left|\left|}
\newcommand{\drdr}{\right|\right|}
\newcommand{\lra}{\left\langle}
\newcommand{\rra}{\right\rangle}

\newcommand{\bal}{\begin{align}}
\newcommand{\eal}{\end{align}}
\newcommand{\beq}{\begin{equation}}
\newcommand{\eeq}{\end{equation}}

\newcommand{\ba}{\begin{align*}}
\newcommand{\be}{\begin{equation*}}
\newcommand{\ee}{\end{equation*}}

\newcommand{\EE}{\mathbb{E}}
\newcommand{\PP}{\mathbb{P}}

\newcommand{\sepa}{\left|\right.}

\newcommand{\crg}{[\![}
\newcommand{\crd}{]\!]}

\newcommand{\sgn}{{\rm Sign}}
\newcommand{\vari}{{\rm Var}}
\newcommand{\cov}{{\rm Cov}}
\newcommand{\poin}[1]{\dot{#1}}
\newcommand{\norm}[1]{\Vert #1\Vert}

\maketitle
\begin{abstract}
In the nonlinear diffusion framework, stochastic processes of McKean-Vlasov type play an important role. In some cases they correspond to processes attracted by their own probability distribution: the so-called self-stabilizing processes. Such diffusions can be obtained by taking the hydrodymamic limit in a huge system of linear diffusions in interaction. In both cases, for the linear and the nonlinear processes, small-noise asymptotics have been emphasized by specific  large deviation phenomenons. The natural question, therefore, is: is it possible to interchange the mean-field limit with the small-noise limit ? The aim here is to consider this question by proving that the rate function of the first particle in a mean-field system converges to the rate function of the hydrodynamic limit as the number of particles becomes large. 
\end{abstract}
\medskip

{\bf Key words and phrases:} Inversion of limits ; Large deviations ; Interacting particle systems ; Hydrodynamic limit ; Nonlinear  diffusions ; small-noise limit \par\medskip

\par\medskip

\section{Introduction}

In the stochastic convergence framework, the large deviation theory plays an essential role for describing the rate at which the probability of certain rare events decays. Each convergence result therefore leads to find the large deviation rate associated with. In suitable cases, the knowledge of the so-called \emph{large deviation principle (LDP)} even permits to obtain informations about the convergence itself (see the central limit theorem \cite{Bryc}). 

This paper is concerned with the convergence of continuous stochastic processes defined as small random perturbations of dynamical systems. In the classical diffusion case, the stochastic process converges in the small-noise limit to the deterministic solution of the dynamical system and the large deviation theory developed by Freidlin and Wentzell \cite{FW98} emphasizes the behaviour of the rare event probabilities. More recently, Herrmann, Imkeller and Peithmann \cite{HIP} studied the large deviation phenomenon associated with the McKean-Vlasov process, a particular nonlinear diffusion which is attracted by its own law (the so-called self-stabilizing effect). This process appears for instance in the probabilistic interpretation of the granular media equation. They presented the explicit expression of the rate function $J_\infty$ and the Kramers' rate which is related to the time needed by the diffusion to exit a given bounded domain.  

The aim of this paper is to better understand the link between the large deviation principle of the nonlinear diffusion and the classical theory developed by Freidlin and Wentzell. More precisely, the McKean-Vlasov equation describes the behaviour of one  particle in a huge system of particles in interaction, as a result of the hydrodynamic limit in a mean-field system. The natural question, therefore, is to emphasize the link between the rate function (or entropy function) $J_\infty$ of the nonlinear diffusion and the Freidlin-Wentzell rate function $J_N$ associated to one particle in a mean-field system of size $N$. We prefer to use functional analysis tools rather to develop the probabilistic interpretation of the corresponding equations. 

The material is organized as follows: first we discuss and recall different notions associated with the large deviation theory. Secondly we present the model and point out the link between nonlinear diffusion and high dimensional classical diffusions: the so-called mean-field effect. The third section will be devoted to the main result: the convergence of the rate functions $J_N\to J_\infty$ as $N$ becomes large. Finally, we present some immediate consequences and a generalization result.
\subsection{A large deviation principle}
Let us introduce the large deviation theory using some simple arguments. We consider a probability space $\left(\Omega,\mathcal{F},\PP\right)$ and $\left(X_k\right)_{k\in\bNb^*}$ a sequence of independent and identically distributed random variables. This sequence is concerned with several classical convergence results: the strong law of large numbers points out that the arithmetic average 
\(
\overline{X}_n:=\frac{1}{n}\sum_{k=1}^n X_k\,
\)
converges almost surely to the mean $\EE[X_1]$ as $n$ goes to infinity. The Central Limit Theorem goes further providing the distribution around this limiting value. Indeed, the random variable
\(
\sqrt{n}\left(\overline{X}_n-\EE[X_1]\right)
\)
converges in distribution to the centered gaussian law with variance ${\rm Var}(X_1)$. Let us note that we do not specify the hypotheses required for these two results to occur. The idea of the large deviations is to go even further estimating the probability of rare events: typically, the probability for the empirical mean $\overline{X}_n$ to be far away from $\mathbb{E}[X_1]$ or the probability that the empirical measure  $\frac{1}{n}\sum_{k=1}^n\delta_{X_k}$ is far from  $\mathbb{P}_{X_1}$, the probability distribution of $X_1$.

In order to measure how small the probability of a rare event is, it is convenient to describe the distribution of the vector $(X_1,\ldots,X_n)$ and to prove that it is concentrated on a small set of typical values with high probability (see \cite{DZ} for precise statements). Let us  illustrate this important feature by an example. We assume that $X_1$ is a finite valued random variable  $\#X_1(\Omega)=d>0$. Without loss of generality, we set $X_1(\Omega)=M:=\left\{1,\ldots,d\right\}$ and denote $p_i:=\PP(X_1=i)$ for any $1\leq i\leq d$. $M^n$ therefore corresponds to the family of sequences (messages) of length $n$. The main interest in the study of rare events is to define the entropy of the so-called \emph{typical messages} and the following surprising remark holds: the probability of being a typical message goes to $1$ as $n$ goes to infinity despite the number of such messages is negligible with respect to the number $d^n=\#M^n$ of all possible sequences.

To make this remark precise, let us define the entropy of the probability distribution $\left(p_1,\ldots,p_d\right)$ as follows
\be
H(p):=H\left(p_1,\ldots,p_d\right):=-\sum_{k=1}^d p_k\log(p_k)\,.
\ee
For a given positive real $\epsilon$, we introduce the set of \emph{typical messages}:
\be
\mathcal{T}_n^\epsilon:=\left\{\left(i_1,\ldots,i_n\right)\in M^n\,\,:\,\,e^{-n\left(H(p)+\epsilon\right)}\leq p_{i_1}\times\ldots\times p_{i_n}\leq e^{-n\left(H(p)-\epsilon\right)}\right\}\,.
\ee
Using the law of large numbers applied to the sequence of independent  
and identically distributed random variables $(\log p_{X_i})_{1\le i\le n}$, the following properties yields  
\[
\lim_{n\to\infty}\PP\left(\left(X_1,\cdots,X_n\right)\in \mathcal{T}_n^\epsilon\right)=1\quad \mbox{and}\quad
\# \mathcal{T}_n^\epsilon\leq e^{n\left(H(p)+\epsilon\right)}\,.
\]
These two results occur for any $\epsilon>0$; in particular, if $H(p)<\log(d)$, we highlight a set of messages $\mathcal{T}_n^\epsilon$ whose probability is close to $1$ for large $n$ whereas its size is small compared to the whole space: $\# \mathcal{T}_n^\epsilon=o\left(d^n\right)=o\left(\#M^n\right)$. In other words, the trajectory $\left(X_1,\cdots,X_n\right)$ has a small probability to be outside a small part of the phase space $M^n$. This discussion is based on the explicit expression of the entropy function which permits to describe the probability of paths deviating from the typical message ones (large deviation phenomenon).\\

In this paper, the framework concerns continuous time processes depending on a parameter $\sigma$ and we describe the behaviour of this family in the small-parameter limit. Even if the state space is infinite, this idea is similar to the above discussion, we need to find out a rate function (entropy) which describes the probability of a trajectory to be far away from typical paths. 

 Let us consider a family of continuous stochastic processes $X^\sigma:=\left(X_t^\sigma\right)_{t\in[0,T]}$ with $T<\infty$. In the following, the family of processes $\left(X^\sigma\right)_{\sigma>0}$ is said to satisfy a \emph{large deviation principle} if there exists a lower semi-continuous mapping (called \emph{rate function}) $I$ from $\mathcal{C}\left([0,T];\bRb^d\right)$ to $\bRb_+$ such that
\be
\limsup_{\sigma\to0}\frac{\sigma^2}{2}\log\left[\PP\left\{X^\sigma\in\mathbb{F}\right\}\right]\leq-\inf_{\varphi\in\bFb}I(\varphi)
\ee
for any closed subset $\mathbb{F}\subset\mathcal{C}\left([0,T];\bRb^d\right)$ equipped with the uniform topology and
\be
\liminf_{\sigma\to0}\frac{\sigma^2}{2}\log\left[\PP\left\{X^\sigma\in\mathbb{G}\right\}\right]\geq-\inf_{\varphi\in\bGb}I(\varphi)
\ee
for any open subset $\mathbb{G}\subset\mathcal{C}\left([0,T];\bRb^d\right)$. $I$ is a \emph{good} rate function if its level sets are compact subsets of $\mathcal{C}\left([0,T];\bRb^d\right)$. \\
We now focus our attention to stochastic differential equations driven by a Brownian motion. Schilder's theorem deals with the LDP of $X^\sigma:=\left(\sigma B_t\right)_{t\in[0,T]}$, where $B$ is a standard $d$-dimensional Wiener process (see Theorem 5.2.3 in \cite{DZ}). The associated good rate function is given by 
\be
I_0(\varphi):=\frac{1}{4}\int_0^T \norm{\poin{\varphi}(t)}^2dt\,,
\ee
if $\varphi$ belongs to the set of absolutely continuous functions starting in $0$, denoted by $\cHc_0$. If $\varphi\notin\cHc_0$, we set $I_0(\varphi):=+\infty$. Here $\norm{\cdot}$ stands for the euclidean norm in $\bRb^d$. The study elaborated by Schilder permits to go further in the description of LDP for diffusions as presented by Freidlin and Wentzell. If $X^\sigma$ satisfies the stochastic differential equation:
\be
X^\sigma_t=x+\sigma B_t-\int_0^tb(s,X_s^\sigma)ds\,,
\ee
where the drift term $b(t,x)$ is a continuous function with respect to the time variable and locally Lipschitz with respect to the space variable; then the family $\left(X^\sigma\right)_{\sigma>0}$ admits a LDP with the good rate function
\be
I_b(\varphi):=\frac{1}{4}\int_0^T\norm{\poin{\varphi}(t)+b(t,\varphi(t))}^2\, dt
\ee
for $\varphi\in\cHc_x$ (the set of absolutely continuous functions starting in $x$). For $\varphi\notin\cHc_x$, $I_b(\varphi):=+\infty$. Let us focus our attention to the \emph{typical paths } of such a diffusion. In fact, in the particular case of a deterministic equation
\begin{equation}
\label{eq:boucle}
\Psi_t(x)=x-\int_0^t b(s,\Psi_s(x))\,ds
\end{equation}
admitting a unique solution, the diffusion $X^\sigma$ starting in $x$ converges in probability towards the deterministic trajectory $\Psi(x)$ in the small-noise limit. The Freidlin-Wentzell LDP estimates the rate of convergence: introducing \[
\bFb:=\left\{\varphi\in\cCc\left([0,T];\bRb^d\right):\ \Vert\varphi-\Psi(x)\Vert_\infty\geq\delta\right\},
\]
where $\Vert\cdot\Vert_\infty$ stands for the uniform norm, we obtain
\be
\limsup_{\sigma\to0}\frac{\sigma^2}{2}\log\PP\left(X^\sigma\in\bFb\right)
\leq-\inf_{\varphi\in\bFb}I_b(\varphi)<0\,.
\ee
Let us finally note that the precise description of the deviation phenomenon permits to deal with the small-noise asymptotics of exit times $\tau_\cDc$ from a domain of attraction $\cDc$. Namely if the drift term of the diffusion is in the so-called gradient case, that is $b(t,x)=\nabla V(x)$, if moreover $V$ reaches a local minimum for $x=a$ and $\cDc$ is a bounded domain of attraction associated to $a$, then a Kramers' type law can be observed. A weak version of this result is the following asymptotic expression:
\[
\lim_{\sigma\to 0}
\frac{\sigma^2}{2}\log \EE[\tau_\cDc]=\inf_{T>0}
\inf_{\substack{\varphi(0)=a,\\ \varphi(T)\in\partial\cDc}}I_{\nabla V}(\varphi)=\inf_{y\in\partial \cDc}V(y)-V(a)
\]
In other words, not only is the rate function a key tool for the description of the diffusion deviation from typical trajectories (linked to a study on a fixed time interval $[0,T]$), but it is also involved in the description of exit times from a domain (a study developed on the whole time axis).

The aim of our paper, therefore, is to describe some nice properties of the rate function, not in the classical diffusion case just described above, but for self-stabilizing diffusions of the McKean-Vlasov type, diffusions attracted by their own law.
Let us finally note that for other applications of large deviations to communication, optic and biology, we refer the reader to \cite{DZ}.

\subsection{The self-stabilizing model}
From now on, we restrict the study to the McKean-Vlasov model: for $x\in\mathbb{R}^d$, the process satisfies the following stochastic differential equation:
\begin{equation}
\label{eq:intro:init}
\left\{\begin{array}{l}
X_t^\sigma=x+\sigma B_t-\int_0^t\nabla W_s^\sigma\left(X_s^\sigma\right)ds\\[5pt]
W_t^\sigma:=V+F\ast u_t^\sigma.      
\end{array}\right.
\end{equation}
The $*$ symbol stands for the convolution product and $u_t^\sigma$ denotes the density of the probability distribution $\mathbb{P}_{X^\sigma_t}$. Since the own law of the process plays an important role in the structure of the drift term, this equation is nonlinear, in the sense of McKean, see for instance \cite{McKean,McKean1966}. Three terms contribute to the infinitesimal dynamics. \begin{itemize}\item The first one is the noise generated by the $d$-dimensional Brownian motion $(B_t,\, t\ge 0)$. 
\item The second force is related to the attraction between one trajectory $t\mapsto X_t^\sigma(\omega_0)$, $\omega_0\in\Omega$, and the whole set of trajectories. Indeed, we observe:
\be
\nabla F\ast u_t^\sigma\left(X_t^\sigma(\omega_0)\right)=\int_{\omega\in\Omega}\nabla F\left(X_t^\sigma(\omega_0)-X_t^\sigma(\omega)\right){\rm d}\mathbb{P}\left(\omega\right)\,.
\ee
Consequently, $F$ is called the \emph{interaction potential}. The interaction only depends on the difference $X^\sigma_t(\omega_0)-X_t^\sigma(\omega)$ and therefore can be associated to the convolution product. Let us note that other dependences have been studied namely the quantile case: the drift is then a continuous function of the quantile of the distribution $\mathbb{P}_{X_t^\sigma}$, see \cite{kolokol}.
\item The last term corresponds to the function $V$, the so-called \emph{confining potential}. The solution $X^\sigma_t$ roughly represents the motion of a Brownian particle living in a landscape $V$ and whose inertia is characterized by $F$. Therefore it is easy to imagine that the minimizers of the potential $V$ attract the diffusion if $F(0)=0$. 
\end{itemize}
Let us now present the hypotheses concerning the functions $F$ and $V$. The confining potential $V$ satisfies:
\begin{itemize}
\item[{\bf (V1)}]\emph{$V$ is a $\cCc^2$-continuous function.}
\item[{\bf (V2)}]\emph{$\nabla^2 \,V(x)\geq0$ for all $x\notin K$ where $K$ is a compact subset of $\mathbb{R}^d$.}
\end{itemize}
Combining (V1) and (V2) ensures the  existence of a solution to \eqref{eq:intro:init}. The interaction function satisfies:
\begin{itemize}
\item[{\bf (F1)}] \emph{There exists a function $G:\bRb_+\to\bRb_+$ such that $F(x)=G(\norm{x})$.}
\item[{\bf (F2)}] \emph{$G$ is an even polynomial function with $\deg(G)\geq2$ and $G(0)=0$.}
\item[{\bf (F3)}] \emph{The following asymptotic property holds  $\displaystyle\lim_{r\to+\infty}G(r)=+\infty$.}
\end{itemize}

Let us now complete the description of this McKean-Vlasov model by briefly recalling several already known results concerning \eqref{eq:intro:init}. 
\begin{itemize}
\item \emph{Probabilistic interpretation of PDEs.} In fact, the self-stabilizing diffusion corresponds to the probabilistic interpretation of the granular media equation. The probability density function of $X^\sigma_t$, starting at $x$, is represented by $(t,x)\mapsto u^\sigma_t(x)$ and  satisfies the following partial differential equation 
\beq
\label{granular}
\frac{\partial}{\partial t}u_t^\sigma={\rm div}\left\{\frac{\sigma^2}{2}\nabla_x u_t^\sigma+u_t^\sigma\left(V+F\ast u_t^\sigma\right)\right\}\,.
\eeq
This equation is strongly nonlinear since it contains a \emph{quadratic} term of the form $u_t^\sigma (F\ast u_t^\sigma)$. This link between the granular media equation \eqref{granular} and the McKean-Vlasov diffusion \eqref{eq:intro:init} permits to study the PDE by  probabilistic methods \cite{CGM,Malrieu2003,Funaki1984}.
\item \emph{Existence and uniqueness.}
The existence and the uniqueness of a strong solution $X^\sigma$ to \eqref{eq:intro:init} defined on $\mathbb{R}_+$ has been proven in \cite{HIP} (Theorem 2.13). Moreover the long-time asymptotic behaviour of the probability distribution $\mathbb{P}_{X^\sigma_t}$ has been studied in \cite{CGM,BRV} (for convex functions $V$) and in \cite{T2010,T2011e} for the non-convex case. In this second case, the key of the proofs essentially consists in using the results of \cite{HT1,HT2,HT3} about the non-uniqueness of the invariant probabilities (that means in particular that there exist several positive stationary solutions of the granular media equation \eqref{granular} which have a total mass equal to $1$).
\item \emph{Large deviation principle.} The noise intensity appearing in the equation \eqref{eq:intro:init} is parametrized by $\sigma$. The aim of the large deviation principle is to describe precisely the behaviour of the paths in the small-noise limit. For any $T>0$, we can prove, see \cite{HIP}, that the family of processes $\left(X^\sigma\right)_{\sigma>0}$  satisfies a large deviation principle with the following good rate function $J_\infty$:
\begin{equation}\label{eq:def:Jinf}
\displaystyle J_\infty(f):=
\frac{1}{4}\int_0^T\norm{\poin{f}(t)+\nabla V\left(f(t)\right)+\nabla F\left(f(t)-\Psi_\infty^x(t)\right)}^2\,dt,
\end{equation}
if $f\in\cHc_x$
and otherwise $J_\infty(f):=+\infty$. Here the function $\Psi_\infty^x$ is independent of $F$ and satisfies the following ordinary differential equation:
\begin{equation}
\label{Psi}
\Psi_\infty^x(t)=x-\int_0^t\nabla V\left(\Psi_\infty^x(s)\right)ds,\quad\,x\in\bRb^d\,.
\end{equation}
In other words, the diffusion process $(X^\sigma_t,\, t\ge 0)$ converges exponentially fast towards the deterministic solution $\Psi_\infty^x$ as $\sigma$ tends to $0$. The limit function for a classical diffusion $Y^\sigma_t$ defined by
\[
Y_t^\sigma=x+\sigma B_t-\int_0^t \nabla V(Y^\sigma_s)\,ds
\]
is exactly the same: the self-stabilizing phenomenon does not change the limit, it only changes the speed of convergence. Indeed the rate function $J_\infty$ clearly depends on $F$. If the function $F$ is convex, the trajectories of the McKean-Vlasov diffusion $X^\sigma$ are closer to $\Psi_\infty^x$ than the ones of the diffusion $Y^\sigma$.

Since the asymptotic behaviour has been described on a fixed-time interval $[0,T]$, the next step is  to describe the asymptotic behaviour on the whole time interval and namely the study of exit problems: the first time the self-stabilizing diffusion exits from a given bounded domain. This problem has already been solved if both $V$ and $F$ are uniformly strictly convex functions, see \cite{HIP,T2011f} by the use of large deviation technics. In \cite{T2011f}, the method is based on the exit problem for an associated mean-field system of particles.
\end{itemize}

\subsection{An interacting particle system}
The McKean-Vlasov diffusion $X^\sigma$, described in the previous section, corresponds to the movement of a particle in a continuous mean-field system in the so-called hydrodynamical limit, that is, as the number of particles tends to infinity. The mean-field system associated to the self-stabilizing process \eqref{eq:intro:init} is a $N$ dimensional random dynamical system $(X^{i,N,\sigma})_{1\le i\le N}\in\otimes^N\mathcal{C}([0,T];\mathbb{R}^d)$ satisfying
\begin{align}
\label{eq:intro:meanfield}
dX_t^{i,N,\sigma}=\sigma \,dB_t^i-\nabla V(X_t^{i,N,\sigma})\,dt
-\frac{1}{N}\sum_{j=1}^N\nabla F(X_t^{i,N,\sigma}-X_t^{j,N,\sigma})\,dt
\end{align}
and $X_0^{i,N,\sigma}:=x$. Here, $(B_t^i)_{t\in\mathbb{R}_+}$ stands for a family of $N$ independent $d$-dimensional Brownian motions. We also assume $B^1=B$, in other words, both diffusions $X^{1,N,\sigma}$ and $X^\sigma$ (see \eqref{eq:intro:init}) are defined with respect to the same Wiener process (this is possible due to the existence of a strong solution). The \emph{propagation of chaos} then permits to link  \eqref{eq:intro:init} and \eqref{eq:intro:meanfield}. It is essentially based on the following intuitive remark. The larger $N$ is, the less influence a given particle $X^{j,N,\sigma}$ has on the first particle $X^{1,N,\sigma}$. Consequently, it is reasonable to consider that the particles become less and less dependent as the number of particles becomes large. The empirical measure $\frac{1}{N}\sum_{j=1}^N\delta_{X_t^{j,N,\sigma}}$ therefore converges towards a measure $\mu_t^\sigma$ which corresponds to the own distribution of $X_t^{1,\infty,\sigma}$. If fact this law corresponds to $\mathbb{P}_{X^ \sigma_t}$. For a rigorous proof of this statement, see \cite{Sznitman,M1996}. It is also possible to adapt a coupling result developed for instance in \cite{BRTV} in order to obtain the following convergence:
\be
\lim_{N\to\infty}\EE\left\{\sup_{0\leq t\leq T}\norm{X_t^{1,N,\sigma}-X_t^\sigma}^2\right\}=0.
\ee
\emph{Large deviation principle.} For $N$ large, the diffusion $X^\sigma$ defined in \eqref{eq:intro:init} is close to the diffusion $X^{1,N,\sigma}$ defined in \eqref{eq:intro:meanfield}. Then it is of particular interest to know if these two diffusions have the same small-noise asymptotic behaviour. The large deviations associated with \eqref{eq:intro:meanfield} are quite classical since the system of particles is a Kolmogorov diffusion of the form
\begin{equation}\label{eq:kolmo}
dX^{i,N, \sigma}_t=\sigma dB_t^i-
N\times \nabla_{x_i}\Upsilon^N(X^{1,N,\sigma}_t,\ldots,X^{N,N,\sigma}_t)\,dt,\quad X^{i,N, \sigma}_0=x,
\end{equation}
with the following potential:
\begin{align*}
&\Upsilon^N(z_1,\cdots,z_N):=\int_{\mathbb{R}^d}V(x)\mu^N(dx)+\frac{1}{2}\int_{\mathbb{R}^d\times\mathbb{R}^d}F(x-y)\mu^N(dx)\mu^N(dy).
\end{align*}
Here $\mu^N:=\frac{1}{N}\sum_{j=1}^N\delta_{z_j}$.
This approach is directly linked to the particular form of the interaction function $F$ which only depends on the norm, see hypothesis (F1) and (F2).
The good rate function, associated with the uniform topology, is a functional defined by
\begin{equation*}
I^N(\Phi):=\frac{1}{4}\int_0^T\norm{\poin{\Phi}(t)+N\nabla\Upsilon^N\left(\Phi(t)\right)
}^2dt,
\end{equation*}
if $\Phi:[0,T]\to \left(\bRb^d\right)^N$ is an absolutely continuous function with the initial condition $\Phi(0)=\overline{x}:=(x,\ldots,x)$ and $I^N(\Phi):=+\infty$ otherwise. The rate function $I^N$ can therefore be rewritten in this way: if $\Phi:=\left(f_1,\cdots,f_N\right)$, we obtain
\begin{equation}
\label{IN}
I^N\left(\Phi\right)=\frac{1}{4}\sum_{i=1}^N\int_0^T\norm{\poin{f_i}(t)+\nabla V\left(f_i(t)\right)+\frac{1}{N}\sum_{k=1}^N\nabla F\left(f_i(t)-f_k(t)\right)}^2dt
\end{equation}
if $f_i\in\mathcal{H}_x$ for any $1\leq i\leq N$. If one function of the family $(f_i)_{1\le i\le N}$   does not belong to $\cHc_x$, then we set $I^N\left(f_1,\cdots,f_N\right):=+\infty$. Let us just note that this LDP leads to the description of the exit problem for the McKean-Vlasov system \cite{T2011f}.
Since a LDP holds for the whole particle system, a LDP in particular holds for the first particle $(X^{1,N,\sigma})$ with the good rate function $J_N$ obtained by projection:
\begin{equation}
\label{eq:def:JN}
J_N(f):=\inf_{f_2,\cdots,f_N\in\mathcal{H}_x}I^N\left(f,f_2,\cdots,f_N\right).
\end{equation}
Since $X^{1,N,\sigma}$ is close to the self-stabilizing process $X^\sigma$, solution of \eqref{eq:intro:init}, we aim to state that the functional $J_N$ converges towards $J_\infty$, the entropy function of the mean-field diffusion, as $N$ becomes large. In other words, is it possible to interchange the limiting operations concerning the asymptotic small noise $\sigma$ and the asymptotic large number of particles $N$ i.e. the hydrodynamic limit ?


\section{Convergence of the rate functions}
In this section, we emphasize the main result of this study. We prove that the large deviation rate function $J_N$ associated to the first particle in the huge McKean-Vlasov system of particles in interaction is close to the rate function of the self-stabilizing (nonlinear) diffusion. 
\begin{thm}\label{thm1}
Let $x\in\mathbb{R}^d$. Under Hypotheses {\rm (V1)--(V2)} and {\rm (F1)--(F3)}, the rate function $J_N$ defined by \eqref{eq:def:JN} converges towards $J_\infty$ defined by \eqref{eq:def:Jinf}, as $N$ tends to infinity. Let $f$ be an absolutely continuous function from $[0,T]$ to $\bRb^d$ such that $f(0)=x$, then
\begin{align*}
\lim_{N\to+\infty}\,J_N(f)=J_\infty(f).
\end{align*}
Moreover the convergence is uniform with respect to $f$ on any compact subset of $\mathcal{C}([0,T],\mathbb{R}^d)$ endowed with the uniform topology.
\end{thm}
\begin{proof}
{\bf Step 1.} Let us first prove (easiest part) the upper-bound
\begin{equation}
\label{eq:limite}
\limsup_{N\to+\infty}J_N(f)\leq J_\infty(f)\,.
\end{equation}
By definition, $J_N(f)\leq I^N(f,f_2,\cdots,f_N)$ for any $f_2,\cdots,f_N\in\mathcal{H}_x$ where $I^N$ is defined by \eqref{IN}. Hence, we can choose $f_k:=\Psi_\infty^x$ for all $2\leq k\leq N$. Let us remind the reader that $\Psi_\infty^x$ is the solution of \eqref{Psi}. Thus we obtain:
\begin{align*}
J_N(f)&
\leq I^N(f,\Psi_\infty^x,\cdots,\Psi_\infty^x)\\
&\leq\frac{1}{4}\int_0^T\norm{ \poin{f}(t)+\nabla V(f(t))+\Big(1-\frac{1}{N}\Big)\nabla F\left(f(t)-\Psi_\infty^x(t)\right)}^2\,dt\\
&+\frac{N-1}{4}\int_0^T\norm{\poin{\Psi}_\infty^x(t)+\nabla V\left(\Psi_\infty^x(t)\right)+\frac{1}{N}\nabla F\left(\Psi_\infty^x(t)-f(t)\right)}^2\,dt\,.
\end{align*}
By definition of $\Psi_\infty^x$, we have $\poin{\Psi}_\infty^x+\nabla V\left(\Psi_\infty^x\right)=0$. The previous inequality yields:
\begin{align*}
J_N(f)&\leq\frac{1}{4}\int_0^T\norm{\poin{f}(t)+\nabla V(f(t))+\Big(1-\frac{1}{N}\Big)\nabla F\left(f(t)-\Psi_\infty^x(t)\right)}^2\,dt\\
&+\frac{1}{4N}\int_0^T\norm{\nabla F\left(\Psi_\infty^x(t)-f(t)\right)}^2\,dt.
\end{align*}
Taking the limit as $N$ goes to infinity in the previous inequality leads to the announced upper-bound \eqref{eq:limite}. Let us just note that, due to the local Lipschitz property of the interaction function $\nabla F$, this convergence is uniform with respect to $f$ on any compact set for the uniform topology. \\[5pt]
{\bf Step 2.} Let us focus our attention to the lower bound:
\begin{equation}\label{eq:lower}
\liminf_{N\to+\infty}J_N(f)\geq J_\infty(f)\,.
\end{equation}
{\bf Step 2.1} Let us recall that $J_N$ is defined as a minimum \eqref{eq:def:JN} and let us prove that it is reached: there exists $(f_2^*,\ldots,f_N^*)\in \cHc_x^{N-1}$ such that
 \[
J_N(f)=I^N(f,f_2^*,\ldots,f_N^*).
\]
If we consider a function $g\in \cHc_x$ then the function $\overline{g}$ defined by 
$\overline{g}(t):=g(t)-x$ for all $t\in[0,T]$ belongs to $\cHc_0$ which is an Hilbert space endowed with the usual norm $\Vert \overline{g}\Vert_H^2:=\int_0^T\norm{\poin{g}(t)}^2\,dt$. Let us introduce now the Hilbert space $\cHc_0^{N-1}$ with the norm $\norm{(g_2,\ldots,g_{N})}^2:=\sum_{k=2}^{N}\Vert g_k\Vert_H^2$. Due to the regularity of both the interaction potential $F$ and the confining potential $V$, it is quite simple to prove that
\[
(\overline{g}_2,\ldots,\overline{g}_{N})\mapsto I^N(f,g_2,\ldots,g_{N})
\]
is a continuous function in the Hilbert space $\cHc_0^{N-1}$ (the details are left to the reader).

Since $J_N(f)$ is the minimum, then for any $\epsilon>0$, there exist $f_2^\epsilon,\cdots,f_N^\epsilon$ belonging to $\cHc_x$ such that
\be
I^{N}\left(f,f_2^\epsilon,\cdots,f_N^\epsilon\right)\leq J_N\left(f\right)+\epsilon\,.
\ee
Let us consider the set $\mathcal{S}_f^x\subset\mathcal{H}_0^{N-1}$ of functions $(\overline{g}_2,\cdots,\overline{g}_N)$ satisfying 
\begin{equation}\label{eq:eq}
I^{N}(f,g_2,\cdots,g_N)\leq 2J_\infty \left(f\right). 
\end{equation}
By \eqref{eq:limite}, for $N$ large enough and $\epsilon$ small enough, we obtain that 
$(\overline{f_2^\epsilon},\cdots,\overline{f_N^\epsilon})\in \mathcal{S}_f^x$. Moreover let us prove that the subset  $\mathcal{S}_f^x$ is included in a closed ball of $\cHc_0^{N-1}$. Indeed the following inequality holds
\begin{align}
\label{eq:mino}
I^{N}\left(f,g_2,\cdots,g_N\right)&
\geq\frac{1}{4}\sum_{i=1}^N\int_0^T\norm{ \poin{g}_i(t)}^2\,dt+R_1(f,g,N)+R_2(f,g,N)
\end{align}
with
\begin{align*}
\left\{\begin{array}{l}
R_1(f,g,N):=\displaystyle\frac{1}{2}\sum_{i=1}^N\int_0^T\lra \poin{g}_i(t)\vert\nabla V\left(g_i(t)\right)\rra dt\\
R_2(f,g,N):=\displaystyle \frac{1}{2N}\sum_{i=1}^N\sum_{k=1}^N\int_0^T\lra \poin{g}_i(t)\vert\nabla F\left(g_i(t)-g_k(t)\right)\rra\, dt,\end{array}\right.
\end{align*}
and the convention $g_1:=f$. Here $\lra\cdot\vert\cdot\rra$ stands for the Euclidian scalar product in $\mathbb{R}^d$. 
We first observe that Hypothesis (F1) leads to
\begin{align*}
R_2(f,g,N)&=\frac{1}{4N}\sum_{i=1}^N\sum_{k=1}^N\int_0^T\lra \poin{g}_i(t)-\poin{g}_k(t)\vert\nabla F\left(g_i(t)-g_k(t)\right)\rra\, dt.\\
&=\frac{1}{4N}\sum_{i=1}^N\sum_{k=1}^N F(g_i(T)-g_k(T))\ge \frac{N}{4}\,\inf_{z\in\mathbb{R}^d}F(z).
\end{align*}
Due to the hypothesis on the interaction function $F$ the right hand side of the previous inequality is finite. With similar arguments, we get
\be 
R_1(f,g,N)=\frac{1}{2}\sum_{i=1}^N\Big(V(g_i(T))-V(x)\Big)\ge \frac{N}{2}\, \inf_{z\in\mathbb{R}^d}V(z)-\frac{N}{2}\, V(x).
\ee 
These two previous inequalities combined with \eqref{eq:mino} and \eqref{eq:eq} permit to prove the existence of a constant $C(N,x,f)$ only depending on $f$, $x$ and $N$ such that
\begin{align}\label{eq:bou}
\sum_{i=2}^N\int_0^T\norm{\poin{g}_i(t)}^2\,dt
\leq C(N,x,f).
\end{align}
We immediately deduce that the subset $\mathcal{S}_f^x$ is included in a closed ball of the Hilbert space $\cHc_0^{N-1}$. Since $(\overline{f_2^\epsilon},\ldots,\overline{f_N^\epsilon})\in \mathcal{S}_f^x$, it is possible to extract a subsequence $(\overline{f_2^{\epsilon_n}},\ldots,\overline{f_N^{\epsilon_n}})_{n\ge 0}$ which converges in the weak topology towards a limiting function $(\overline{f_2^{*}},\ldots,\overline{f_N^{*}})\in \cHc_0^{N-1}$. Finally, due to the continuity of the function $I_N$ we deduce that:
\[
I^N(f,f_2^*,\ldots,f_N^*)=\lim_{n\to\infty}I^N(f,
f_2^{\epsilon_n},\ldots,f_N^{\epsilon_n})\le J_N(f)+\lim_{n\to\infty}\epsilon_n=J_N(f).
\]
The minimum $J_N$ is then reached for $(f_2^*,\ldots,f_N^*)\in\cHc_x^{N-1}$.\\[5pt]
{\bf Step 2.2} In order to compute $J_N$, let us point out particular properties of the functions $(f_2^*,\ldots,f_N^*)$. Since the minimum is reached, we are going to compute the derivative of $I^N$ with respect to each coordinate. Since we restrict ourselves to the functional space $\mathcal{H}_x$, we take an absolutely continuous function $g\in \cHc_0$ and we consider the following limit
\begin{align*}
D_{2}I^N\left(f,f_2,\ldots,f_N\right)(g)
&:=\lim_{\delta\to 0}\frac{I^N(f,f_2+\delta g,f_3,\ldots,f_N)-I^N(f,f_2,\ldots,f_N)}{\delta}\,.
\end{align*}
That defines the derivative of $I^N$ with respect to the second argument. In a similar way, we can define $D_iI^N$ for any $2\le i\le N$. Let us compute explicitly $D_2I^N$. For any $1\le i\le N$, we set
\beq
\label{frasier}
\xi_i:=\poin{f}_i+\nabla V(f_i)+\frac{1}{N}\sum_{k=1}^N\nabla F(f_i-f_k)\,,
\eeq
with $f_1:=f$ by convention. So we note that
\begin{equation}
\label{eq:lient}
I^N(f,f_2,\ldots,f_N)=\frac{1}{4}\sum_{i=1}^N\int_0^T\norm{\xi_i(t)}^2\,dt.
\end{equation}
We now observe the derivative of the quantity $\xi_i$ with respect to the second function. In other words, introducing $f_2^\delta:=f_2+\delta g$, and defining $\xi_i^{2,\delta}$ like $\xi_i$ just by replacing $f_2$ by $f_2^\delta$ in \eqref{frasier}, we get
\be
\lim_{\delta\to 0}\frac{\xi_i^{2,\delta}-\xi_i}{\delta}=-\frac{1}{N}\,H(F)(f_i-f_2)g\,,\quad \mbox{for}\ i\neq 2,
\ee
and, for $i=2$:
\be
\lim_{\delta\to 0}\frac{\xi_2^{2,\delta}-\xi_2}{\delta}= \poin{g}+H(V)(f_2)g+\frac{1}{N}\sum_{\substack{1\le k \le N\\ k\neq 2}}H(F)(f_2-f_k)g\,.
\ee
Here $H(F)(x)$ represents the Hessian matrix of the function $F$ at the point $x\in\mathbb{R}^d$. From now on, we simplify the notation $D_2I^N(f,f_2,\ldots,f_N)(g)$ and replace it by $D_2I^N$. By \eqref{eq:lient} and the polarization identity, we obtain
\begin{align}\label{solgen}
D_2I^N&=\lim_{\delta\to 0}\frac{1}{4\delta}\,\sum_{i=1}^N\int_0^T \norm{\xi^{2,\delta}_i(t)}^2-\norm{\xi_i(t)}^2\, dt\nonumber \\
&=\lim_{\delta\to 0}\frac{1}{4\delta}\,\sum_{i=1}^N\int_0^T \Big(\norm{\xi^{2,\delta}_i(t)-\xi_i(t)}^2+2\langle\xi_i(t)\vert \xi^{2,\delta}_i(t)-\xi_i(t)\rangle\Big)\, dt\nonumber \\
&=-\frac{1}{2N}\sum_{i=1}^N\int_0^T
\langle \xi_i(t) \vert H(F)(f_i-f_2)g(t) \rangle\,dt\nonumber \\
&+\frac{1}{2}\int_0^T\langle \xi_2(t) \vert  \poin{g}(t)+\mathcal{R}^f_2 g(t)
\rangle  \,dt,
\end{align}
with
\begin{equation}
\label{eq:def:R}
\mathcal{R}^f_2 :=H(V)(f_2)+\frac{1}{N}\,\sum_{i=1}^N H(F)(f_2-f_i).
\end{equation}
Let us assume that $\xi_2$ is regular (let us say continuously differentiable), then by an integration by parts and since $g(0)=0$, we obtain
\ba
D_2I^N
=-\frac{1}{2}\int_0^T\langle \mathcal{E}_2^f(\xi_2,\ldots,\xi_N)(t)\vert g(t)\rangle dt+\frac{1}{2}\langle \xi_2(T)\vert g(T)\rangle,
\end{align*}
where the function $\mathcal{E}_2^f$ is defined by
\begin{align}\label{eq:def:erond}
\mathcal{E}_2^f(\xi_2,\cdots,\xi_N)(t)&:=\poin{\xi}_2(t)-\mathcal{R}^f_2 \xi_2(t)+\frac{1}{N}\sum_{i=1}^N H(F)(f_2-f_i)\xi_i(t)\,.
\end{align}
We proceed in the same way for $D_jI^N(f_1,f_2,\cdots,f_N)(g)$ for any $2\leq j\leq N$, just replacing $2$ by $j$ in \eqref{eq:def:R} and \eqref{eq:def:erond}.
Since the minimum $J_N(f)$ is reached, that is $J_N(f)=I^N(f,f_2^*,\ldots,f_N^*)$ for some functions $(f_2^*,\ldots,f_N^*)\in\cHc_x^{N-1}$ (see Step 2.1), the following expression vanishes: $D_jI^N(f,f_2^*,\ldots,f_N^*)(g)=0$ for $2\le j \le N$ and any function $g$, and therefore $(\xi_2^*,\ldots,\xi_N^*)$ is solution to the system
\begin{equation}
\label{olala}
\left\{\begin{array}{ll}\mathcal{E}_j^{f^*}(\xi_2^*,\ldots,\xi_N^*)=0,\\
\xi_j^*(T)=0\quad\mbox{for any}\,\,2\leq j\leq N.
\end{array}\right.
\end{equation}
where the functions $\mathcal{E}_j^{f^*}$ are defined like $\mathcal{E}_j^f$ in \eqref{eq:def:erond} (respectively $\xi_j^*$ like $\xi_j$ in \eqref{frasier}), we need just to replace $(f,f_2,\ldots,f_N)$ by $(f,f_1^*,\ldots,f_N^*)$.\\
In fact we do not know if $\xi_j^*$, $2\le j\le N$, are regular functions as assumed. We deduce therefore that $(\xi_2^*,\ldots,\xi_N^*)$ is a \emph{generalized solution} of the system \eqref{olala}.\\[5pt] 
{\bf Step 2.3} In Step 2.1, we proved that $(f_2^*,\ldots,f_N^*)$ minimizes the function $I^N$ and, by Step 2.2, that $(\xi_2^*,\ldots,\xi_N^*)$ which can be expressed as a function of $(f_2^*,\ldots,f_N^*)$ -- see \eqref{frasier} -- satisfies a particular system of differential equations, in a generalized sense, namely \eqref{olala}. Of course $(\xi_2^*(t),\ldots,\xi_N^*(t))=(0,\ldots,0)$ is a classical solution (and consequently a generalized solution) of \eqref{olala}. By uniqueness of the \emph{generalized solutions}, we shall obtain that 
\[
\xi_j^*(t)=0\quad \mbox{for a.e.}\quad 0\le t\le T\quad\mbox{and for}\quad 2\le j\le N.
\]
Let us prove this uniqueness property. Let us consider a function $h=(h_2,\ldots,h_N)$ belonging to $\otimes ^{N-1}\mathcal{C}([0,T],\mathbb{R}^d)$, then using the Cauchy-Lipschitz theorem (see for instance Theorem 3.1 in \cite{Hale}), there exists a $\mathcal{C}^1$-solution $(g_2,\ldots,g_N)$ of the following system of equations:
\begin{align}\label{adjoint}
\dot{g}_j(t)=-\mathcal{R}_j^{f^*}g_j(t)-\frac{1}{N}\sum_{k=1}^N H(F)(f_j^*-f_k^*)g_k(t)+h_{j}(t),\quad 2\le j\le N,
\end{align}
with the initial condition $g_j(0)=0$ for any $2\le j\le N$. Here $f_1^*$ stands for $f$ for notational convenience. In particular $g\in \mathcal{H}_0^{N-1}$.
Using results developed in Step 2.2, the function $I^N$ reaches its minimal value for the arguments $(f_2^*,\ldots,f_N^*)$ and consequently:
\[
\sum_{j=2}^N D_jI^N(f,f_2^*,\ldots,f_N^*)(g_j)=0,
\]
where $g_j$ are solutions of \eqref{adjoint}. Combining the expression \eqref{solgen} and \eqref{adjoint} leads to
\begin{equation}\label{eq:zero}
\sum_{j=2}^N\int_{0}^T\langle \xi_j^*(t)|h_{j}(t) \rangle\, dt=0,
\end{equation}
for any continuous functions $(h_j,\ 2\le j\le N)$. Since $f^*$ is in the function space $\mathcal{H}_x^{N-1}$ and since $\xi_j^*$ is related to $f^*$ by \eqref{frasier}, we know that $\xi_j$ is square integrable function. Using Carleson's theorem (see for instance Theorem 1.9 in \cite{Duo}) and \eqref{eq:zero} we deduce that $\xi_j^*(t)=0$ for a.e. $t\in[0,T]$ and for any $2\le j\le N$.\\[5pt]
{\bf Step 2.4} Using the definition of $\xi_j^*$, the previous step permits to obtain that, for any $2\le j\le N$,
\beq
\label{perceval}
\poin{f}_j^*(t)+\nabla V(f_j^*)(t)+\frac{1}{N}\sum_{k=1}^N\nabla F\left(f_j^*-f_k^*\right)(t)=0\quad \mbox{for a.e.} \ t\in[0,T],
\eeq
with the boundary condition $f^*_j(0)=x$.
Applying once again the arguments presented in Step 2.3 leads to the uniqueness of the solutions for \eqref{perceval}. Since the system is symmetric, we get the existence of a $\mathcal{C}^1$-function $\Psi_N^f$ satisfying
\[
f_2^*(t)=\ldots=f_N^*(t)=\Psi_N^f(t)\quad\mbox{for a.e.}\ t\in[0,T],
\]
and, on the time interval $[0,T]$,
\beq
\label{leodagan}
\poin{\Psi}_N^f(t)+\nabla V(\Psi_N^f)(t)+\frac{1}{N}\nabla F(\Psi_N^f-f)(t)=0,\quad\mbox{with}\  \Psi_N^f(0)=x.
\eeq
We just recall that $f_1^*=f$ for notational convenience. Using the definition of $J_N(f)$, we get
\begin{align}\label{eq:JN}
J_N(f)&=I^N(f,f^*_2,\ldots,f^*_N)\nonumber\\
&=\int_0^T\Vert\poin{f}(t)+\nabla V(f(t))+\Big(1-\frac{1}{N}\Big)\nabla F(f(t)-\Psi_N^f(t))\Vert^2\,dt\,.
\end{align}
Since the first order differential equation \eqref{leodagan} can be associated to a Lipschitz constant which does not depend on the parameter $1/N$, the unique solution $\Psi_N^f (t)$ depends continuously on both the parameter $1/N$ and the time variable (see, for instance, Theorem 3.2 p. 20 in \cite{Hale}). Here we consider that $(t,1/N)$ belongs to the compact set $[0,T]\times[0,1]$, consequently $(t,1/N)\mapsto\Psi_N^f(t)$ is uniformly continuous. Moreover, $\Psi_\infty^f=\Psi_\infty^x$. Hence \eqref{eq:JN} implies
\[
\lim_{N\to\infty}J_N(f)=J_\infty(f),
\]
where $J_\infty$ is defined by \eqref{eq:def:Jinf} and \eqref{Psi}. The proof of the lower-bound \eqref{eq:lower} is then achieved. It is quite easy to prove that the convergence is uniform with respect to the function $f$ on any compact set for the uniform topology. 
\end{proof}
We could provide  
the precise rate of convergence because we give better than a simple  
lower-bound in the previous proof. Indeed, we have obtained the exact expression of $J_N(f)$.

\section{Immediate consequences and further results}
Theorem \ref{thm1} emphasizes the link between the large deviation rate function of the self-stabilizing diffusion \eqref{eq:intro:init} and the rate function associated to the mean-field system \eqref{eq:intro:meanfield}. This result is of particular interest since one of the diffusion is nonlinear whereas the second one is linear and therefore well-known. In this section, we present a coupling result concerning these two diffusions and extend Theorem \ref{thm1} to a more general nonlinear model.

Let us first recall the large deviation principle already presented in the introduction. A family of continuous stochastic processes  $\left(X^\sigma\right)_{\sigma>0}$ is said to satisfy a \emph{large deviation principle} for the uniform topology with good rate function $I$ if the level sets of $I$ are compact subsets of $\mathcal{C}\left([0,T];\bRb^d\right)$ and if
\be
\limsup_{\sigma\to0}\frac{\sigma^2}{2}\log\PP\left(X^\sigma\in\mathbb{F}\right)\leq-\inf_{\varphi\in\bFb}I(\varphi)
\ee
for any closed subset $\mathbb{F}\subset\mathcal{C}\left([0,T];\bRb^d\right)$ equipped with the uniform topology and
\be
\liminf_{\sigma\to0}\frac{\sigma^2}{2}\log\PP\left(X^\sigma\in\mathbb{G}\right)\geq-\inf_{\varphi\in\bGb}I(\varphi)
\ee
for any open subset $\mathbb{G}\subset\mathcal{C}\left([0,T];\bRb^d\right)$. 

Theorem \ref{thm1} ensures the convergence of $J_N(f)$ to $J_\infty(f)$ for any continuous function $f$. In order to point out the large deviation principle, we need to precise this convergence on open and closed subsets of continuous functions.
\begin{cor}\label{cor:inf-sup}
For any open or closed subset $\mathbb{O}\subset \mathcal{C}([0,T];\mathbb{R}^d)$ (for the uniform topology) the following convergence holds
\begin{equation}\label{eq:cor}
\lim_{N\to\infty}\inf_{\varphi\in\mathbb{O}}J_N(\varphi)=\inf_{\varphi\in\mathbb{O}}J_\infty(\varphi).
\end{equation}
\end{cor} 
\begin{proof}
Let $\mathbb{O}$ be a non empty set. If  \(
\displaystyle \inf_{\varphi\in\mathbb{O}}J_\infty(\varphi)=+\infty
\)
then $\mathbb{O}\subset\mathcal{H}^c$ where $\mathcal{H}^c$ is the complementary of the set of absolutely continuous functions. Therefore $J_N(\varphi)=+\infty$ for any function $\varphi\in \mathbb{O}$ and \eqref{eq:cor} is obviously satisfied. Let us assume now that \(
\displaystyle \inf_{\varphi\in\mathbb{O}}J_\infty(\varphi)=\alpha<\infty.
\) Denoting by \[
\kappa_\lambda=\{\varphi\in\mathcal{C}([0,T],\mathbb{R}^d):\ J_\infty(\varphi)\le \lambda\}
\]
which is a compact set since $J_\infty$ is a \emph{good} rate function (see \cite{HIP}), we obtain
\[
\inf_{\varphi\in\mathbb{O}}J_\infty(\varphi)=
\inf_{\varphi\in\mathbb{O}\,\cap \kappa_{2\alpha}}J_\infty(\varphi).
\]
In order to conclude the proof, it suffices to apply the convergence of $J_N$ towards $J_\infty$ developed in Theorem \ref{thm1}, which is in fact uniform with respect to $\varphi$ on any compact subset, in particular on the subset $\kappa_{2\alpha}$.
\end{proof}
This first corollary concerns the rate functions. Let us now focus our attention on the associated processes. A nice coupling property can be obtained describing the link between the self-stabilizing diffusion $(X^\sigma_t,\,t\ge 0)$ defined by \eqref{eq:intro:init} and the linear diffusion $(X^{1,N,\sigma}_t,\, t\ge 0)$ defined by \eqref{eq:intro:meanfield}. Since there exists a unique strong solution to each of these two equations, we can construct $X^\sigma$ and $X^{1,N,\sigma}$ on the same probability space $(\Omega,\mathcal{B},\mathbb{P}_x)$. 
\begin{cor}
\label{cor2} Under Hypotheses (V1)--(V2) and (F1)--(F3), for any $x\in\bRb^d$, each element of the family $(X^{1,N,\sigma})_N$ converges in probability towards the diffusion $X^\sigma$ as $\sigma\to 0$, uniformly with respect to the parameter $N$. In particular let $\delta>0$, then for $N$ sufficiently large (resp. $\sigma$ small), there exists a constant $K_\delta (T)>0$ such that
\be
\PP\Big(\sup_{0\leq t\leq T}\Vert X_t^{1,N,\sigma}-X_t^{\sigma}\Vert\geq\delta\Big)\leq e^{-\frac{K_\delta(T)}{\sigma^2}}.
\ee
\end{cor}
Let us just note that this result implies the convergence in distribution of the first particle -- first coordinate -- of the linear mean-field system  \eqref{eq:intro:meanfield} towards the self-stabilizing process. Combining Corollary \ref{cor:inf-sup} and Corollary \ref{cor2} leads to the following statement: for any closed set $\mathbb{F}$,
\[
\lim_{N\to\infty}\limsup_{\sigma\to 0}\frac{\sigma^2}{2}\, \log \mathbb{P}(X^{1,N,\sigma}\in\mathbb{F})\quad \mbox{and}\quad \limsup_{\sigma\to 0}\lim_{N\to\infty}\frac{\sigma^2}{2}\, \log \mathbb{P}(X^{1,N,\sigma}\in\mathbb{F})
\]
have the same upper bound namely: $\displaystyle-\inf_{\varphi\in\mathbb{F}}J_\infty(\varphi)$. 
A similar result holds for open sets $\mathbb{G}$, replacing both the limit inferior by the limit superior and the upper bound by a lower one.

A second remark: Large deviation principles can sometimes be proven directly by the use of coupling bounds. Nevertheless the coupling bound developed in Corollary \ref{cor2} is not strong enough for such implications. Indeed, the family $(X^{1,N,\sigma})_N$ is called an \emph{exponentially good approximation} of $X^\sigma$ if, for any $\delta>0$, 
\begin{equation}
\label{eq:expoapp}
\lim_{N\to\infty}\limsup_{\sigma\to 0}\sigma^2\log\mathbb{P}_x\Big( \sup_{0\le t\le T}\Vert X^{1,N,\sigma}_t-X^\sigma_t\Vert \ge \delta \Big)=-\infty
\end{equation}
(such a large deviation notion has been for instance developed in \cite{DZ}, Definition 4.2.14). For such approximations, the rate function of the limiting process (as $N$ tends to $\infty$) can be obtain as follows:
\[
J_\infty(\varphi):=\sup_{\delta>0}\liminf_{N\to\infty}\inf_{z\in B(\varphi,\delta)}J_N(z),
\]
where $\displaystyle B(\varphi,\delta)=\{z:\ \sup_{0\le t\le T}\Vert z(t)-\varphi(t)\Vert <\delta  \}$. In practice, \eqref{eq:expoapp} is quite difficult to obtain. Such technics were used in order to prove the Freidlin-Wentzell large deviation result for classical diffusions: the process is approximated by an other stochastic process with piecewise constant diffusion and drift terms (see Theorem 5.6.7 in \cite{DZ}). For the large deviation principle associated with the self-stabilizing diffusion developed in \cite{HIP}, an argument of exponentially good approximation is used but it does not concern the approximation of the non linear process by the first particle of the mean field linear system and therefore it does not use \eqref{eq:expoapp}.
\begin{proof}
By definition, the family of processes $\left(X^{\sigma}\right)_{\sigma>0}$ satisfies a large deviation principle associated with the good rate function $J_\infty$. So, for any closed subset $\bFb$ of $\cHc_x$, we have on one hand
\be
\limsup_{\sigma\to0}\frac{\sigma^2}{2}\log\left[\PP\left\{X^{\sigma}\in\mathbb{F}\right\}\right]\leq-\inf_{\varphi\in\bFb}J_\infty(\varphi)\,.
\ee
On the second hand, by Corollary \ref{cor:inf-sup} and for $N$ large enough, we obtain
\be
\limsup_{\sigma\to0}\frac{\sigma^2}{2}\log\left[\PP\left\{X^{1,N,\sigma}\in\mathbb{F}\right\}\right]\leq-\inf_{\varphi\in\bFb}J_N(\varphi)\leq-\frac{3}{4}\inf_{\varphi\in\bFb}J_\infty(\varphi)\,.
\ee
Introducing the particular subset:
\be
\bFb:=\Big\{\varphi\in\cHc_x:\,\,\exists t_0\in[0,T]\ \mbox{s.t.}\  \Vert\varphi(t_0)-\Psi_\infty^x(t_0)\Vert
\geq\frac{\delta}{4}\Big\}\,,
\ee
where $\Psi_\infty^x$ is defined in \eqref{Psi} that is
\(
\Psi_\infty^x(t):=x-\int_0^t\nabla V\left(\Psi_\infty^x(s)\right)ds\,,
\)
we observe that, for $N$ large and $\sigma$ small,
\ba
\PP\Big(\sup_{0\leq t\leq T}\Vert X_t^{1,N,\sigma}-X_t^{\sigma}\Vert\geq\delta\Big)&\leq\PP\left(X^{1,N,\sigma}\in\bFb\right)+\PP\left(X^{\sigma}\in\bFb\right)\leq e^{-\frac{K_\delta(T)}{\sigma^2}}\,,
\end{align*}
where $\displaystyle K_\delta(T):=C\inf_{\varphi\in\bFb}
J_\infty(\varphi)>0$ with $0<C<3/4$.
\end{proof}

In Theorem \ref{thm1}, we only deal with the \emph{gradient case} of the so-called McKean-Vlasov diffusion starting from initial position $x$. Let us now discuss a more general setting by considering the following nonlinear diffusion:
\be
Y_t^\sigma=x+\sigma B_t-\int_0^t\nabla V\left(Y_s^\sigma\right)ds-\int_0^t\int_{\bRb^d}\mathcal{A}
(Y_s^\sigma,y)\,\nu_s^\sigma(dy)\,ds-l(t).
\ee
Here $\mathcal{A}$ is a general two variables $\mathbb{R}^d$-valued function being a vector flow, non necessary gradient and $l$ is a $\cCc^1$-continuous function from $\bRb_+$ to $\bRb^d$. Finally the probability measure $\nu_s^\sigma$ stands for the distribution $\mathbb{P}_{Y_s^\sigma}$. The aim of this discussion does not concern the existence and uniqueness of such equation, so we assume that $V$, $\mathcal{A}$ and $l$ satisfy suitable conditions for the unique solution to exist. Then, it is possible to adapt the arguments developed in  \cite{HIP} in order to prove that  $\left(Y^\sigma\right)_{\sigma>0}$ satisfies a large deviation principle with the associated rate function:
\begin{equation}
\label{lost}
J_\infty(f):=\frac{1}{4}\int_0^T\Vert \poin{f}(t)+\nabla V(f(t))+\mathcal{A}(f(t),\Psi^x(t))+\poin{l}(t)\Vert^2\,dt
\end{equation}
for any function $f\in\mathcal{H}_{x}$ and $J_\infty(f):=+\infty$ otherwise. Here, the function $\Psi^x$ is defined as the unique solution of the ordinary differential equation:
\begin{equation*}
\Psi^x(t)=x-\int_0^t\nabla V\left(\Psi^x(s)\right)ds-\int_0^t\mathcal{A}(\Psi^x(s),\Psi^x(s))\,ds-l(t).
\end{equation*}
The stochastic model $(Y_t^\sigma)$ can also be approximated by a system of interacting particles. In this context, we can develop a statement similar to Theorem \ref{thm1}. The functional $J_\infty$ is effectively the limit as $N$ goes to infinity of the functional $J_N(f)$ defined by
\begin{equation}
\label{les4400}
\inf_{f_2,\ldots,f_N\in\mathcal{H}_x}\frac{1}{4}\sum_{i=1}^N\int_0^T\Vert \poin{f}_{i}(t)+\nabla V(f_i(t))+\frac{1}{N}\sum_{j=1}^N \mathcal{A}(f_{i}(t),f_{j}(t))+\poin{l}(t)\Vert^2\,dt
\end{equation}
with the convention $f_1=f$. Such a result can be proven under suitable assumptions: 
\begin{itemize}
\item the confining potential $V$ satisfies Hypotheses (V1)--(V2).
\item there exists a \emph{lower bounded} $\mathcal{C}^\infty$-function $\mathbb{A}:\mathbb{R}^d\times\mathbb{R}^d\to \mathbb{R}$ such that 
\be
\mathcal{A}(x,y)=\nabla_x\mathbb{A}(x,y)\quad\mbox{and}\quad \inf_{(x,y)\in\mathbb{R}^d\times\mathbb{R}^d}\mathbb{A}(x,y)>-\infty.
\ee
\item $\mathcal{A}$ satisfies a symmetry property: $\mathcal{A}(x,y)=-\mathcal{A}(y,x)$
\end{itemize}
The details of the proof are left to the reader, it suffices to apply the same arguments. Let us just note that the assumptions, just formulated, concerning $\mathcal{A}$  are sufficient in order to get the upper-bound \eqref{eq:bou}, a crucial step for proving the claimed statement.

We end this study pointing out an example of such a diffusion: 
\begin{equation*}
dX_t=\sigma dB_t-\left(\nabla W\left(X_t\right)-\overline{W}_t\right)dt-\poin{l}(t)dt,
\end{equation*}
where $\overline{W}_t:=\mathbb{E}\left\{\nabla W\left(X_t\right)\right\}$ and $W$ is such that the required conditions are satisfied. This equation actually corresponds to the hydrodynamic limit of an equation characterizing the charge and the discharge of the cathode in a lithium battery (see \cite{DGGHJ,DGH}). In such a framework, $\mathcal{A}(x,y):=\nabla W(x)-\nabla W(y)$ and $\mathbb{A}(x,y):=W(x)-\lra x\vert \nabla W(y)\rra$. Therefore, the rate function can be explicitly computed:
\begin{equation*}
J_\infty(f)=\frac{1}{4}\int_0^T\Vert \poin{f}(t)+\nabla W\left(f(t)\right)-\nabla W\left(x+l(t)\right)+\poin{l}(t)\Vert^2\,dt
\end{equation*}
and is obtained, as announced, as the limit for large $N$ of the rate function:
\[
\inf_{f_2,\ldots,f_N\in\mathcal{H}_x}\frac{1}{4}\sum_{i=1}^N\int_0^T\Vert \poin{f}_{i}(t)+\nabla W(f_i(t))-\frac{1}{N}\sum_{j=1}^N \nabla W(f_{j}(t))+\poin{l}(t)\Vert^2\,dt.
\]
\noindent{\bf Acknowledgements}\\ 
(S.H.) \emph{I would like to thank Dierk Peithmann for interesting discussions concerning this study.} \\ (J.T.): \emph{I would like to thank the Institut de Math\'ematiques de Bourgogne for the time spent there which permits the idea of this article to emerge.\\
I also thank Bielefeld University and especially Barbara  
Gentz.\\
Velika hvala Marini za sve. \'Egalement, un tr\`es grand merci \`a Manue et \`a Sandra pour tout.}

\begin{small}
\def\cprime{$'$}

\end{small}

\end{document}